\documentclass[12pt,leqno]{article}

\usepackage{amssymb,amsmath,amsthm}
\usepackage[all]{xy}
\numberwithin{equation}{section}

\usepackage{enumerate}
\usepackage{mathrsfs}

\newcommand{\nc}{\newcommand}
\nc{\on}{\operatorname}
\newtheorem{theorem}{Theorem}[section]
\newtheorem{proposition}[theorem]{Proposition}
\newtheorem{lemma}[theorem]{Lemma}
\newtheorem{corollary}[theorem]{Corollary}

\theoremstyle{definition}

\newtheorem{definition}[theorem]{Definition}

\newtheorem{example}[theorem]{Example}

\newtheorem{remark}[theorem]{Remark}

\nc{\Prop}{\begin{proposition}}
\nc{\enprop}{\end{proposition}}
\nc{\Lemma}{\begin{lemma}}
\nc{\enlemma}{\end{lemma}}
\nc{\Cor}{\begin{corollary}}
\nc{\encor}{\end{corollary}}
\nc{\Def}{\begin{definition}}
\nc{\enDef}{\end{definition}}
\nc{\Th}{\begin{theorem}}
\nc{\entheorem}{\end{theorem}}

\newcommand{\C}{{\mathbb{C}}}

\newcommand{\R}{{\mathbb{R}}}

\newcommand{\Z}{{\mathbb{Z}}}

\nc{\forl}{[\mspace{-.3mu}[\hbar]\mspace{-.3mu}]}
\nc{\Ls}{(\mspace{-.3mu}(\hbar)\mspace{-.3mu})}

\newcommand{\cor}{{\bf k}}

\nc{\kor}[1][M]{\C_{#1}((\hbar))}
\nc{\koro}[1][M]{\C_{#1}[[\hbar]]}

\nc{\Dh}[1][M]{\shd_{#1}^\hbar}
\nc{\Dhh}[1][M]{\shd_{#1}\forl}
\nc{\Dhhl}[1][M]{\shd_{#1}\Ls}
\nc{\Dhhh}[1][M]{\mathscr{D}_{#1}[\hbar,\opb{\hbar}]}

\nc{\DA}[1][X]{\mathscr{D}^{\mathscr{A}}_{#1}}

\def\phi{{\varphi}}
\def\epsilon{\varepsilon}

\def\sha{\mathscr{A}}

\def\shb{\mathscr{B}}
\def\shc{\mathscr{C}}
\def\shd{\mathscr{D}}

\def\shi{\mathscr{I}}

\def\shl{\mathscr{L}}
\def\shm{\mathscr{M}}

\def\sho{\mathscr{O}}

\newcommand{\ol}{\overline}
\newcommand{\bl}{\bigl(}
\newcommand{\br}{\bigr)}

\newcommand{\lp}{{\rm(}}
\newcommand{\rp}{{\rm)}}

\nc{\RR}{\mathrm{R}}
\nc{\LL}{\mathrm{L}}

\newcommand{\into}{\hookrightarrow}

\newcommand{\A}[1][X]{\mathscr{A}_{{#1}}}

\nc{\Dhl}[1][X]{\shd_{#1}\Ls}
\nc{\Dho}[1][X]{\mathscr{D}_{#1}^\hbar(0)}
\nc{\Ohl}[1][X]{\sho_{#1}\Ls}
\nc{\Oh}[1][X]{\sho_{#1}^\hbar}
\nc{\OOh}[1][X]{\sho_{#1}[[\hbar]]}

\newcommand{\HWo}[1][X]{\widehat{\mathscr{W}}_{#1}(0)}

\newcommand{\gr}{\mathrm{gr}_\hbar}

\newcommand{\Lie}[1][]{\operatorname{\mathsf{L}}\def\temp{#1}
\ifx\temp\empty\else^{(#1)}\fi}

\newcommand{\stkHom}[1][]{\mathfrak{Hom}_{\raise1.5ex\hbox to.1em{}#1}}

\newcommand{\Int}{{\rm Int}}
\newcommand{\loc}{{\rm loc}}

\nc{\oA}[1][X]{\omega_{#1}^{\,\mathscr{A}}}

\nc{\oo}[1][X]{\omega_{#1}}

\newcommand{\Sol}{{\rm Sol}}

\renewcommand{\to}[1][]{\xrightarrow[]{#1}}

\newcommand{\isoto}[1][]{\xrightarrow[#1]%
{{\raisebox{-.6ex}[0ex][-.6ex]{$\mspace{1mu}\sim\mspace{2mu}$}}}}

\newcommand{\Hom}[1][]{\mathrm{Hom}_{\raise1.5ex\hbox to.1em{}#1}}
\newcommand{\RHom}[1][]{\RR\mathrm{Hom}_{\raise1.5ex\hbox to.1em{}#1}}
\newcommand{\Ext}[2][]{\mathrm{Ext}_{\raise1.5ex\hbox to.1em{}#1}^{#2}}
\renewcommand{\hom}[1][]{{\mathscr{H}\mspace{-4mu}om}_{\raise1.5ex\hbox to.1em{}#1}}
\newcommand{\rhom}[1][]{{\RR\mathscr{H}\mspace{-3mu}om}_{\raise1.5ex\hbox to.1em{}#1}}
\newcommand{\ext}[2][]{{\mathscr{E}\mspace{-2mu}xt}_{%
\raise1.5ex\hbox to.1em{}#1}^{#2}}
\newcommand{\BHom}[1][]{\mathrm{Bhom}_{\raise1.5ex\hbox to.1em{}#1}}

\newcommand{\Tens}[1][]{\mathbin{\otimes_{\raise1.5ex\hbox to-.1em{}{#1}}}}
\newcommand{\LTens}[1][]{\mathbin{\otimes_{\raise1.5ex\hbox to-.1em{}#1}^{L}}}
\newcommand{\Tor}[2][]{\mathrm{Tor}^{\raise1.5ex\hbox to.1em{}#1}_{#2}}
\newcommand{\tens}[1][]{\mathbin{\otimes_{\raise1.5ex\hbox to-.1em{}{#1}}}}

\newcommand{\dtens}[1][]%
{{\overset{\mathrm{L}}{\underline{\otimes}}}_{#1}}

\nc{\aut}{{\sha\mspace{-1mu}\mit{ut}}\,}

\newcommand{\Endo}[1][]{\mathrm{End}_{\raise1.5ex\hbox to.1em{}#1}}

\newcommand{\Aut}[1][]{\mathrm{Aut}_{\raise1.5ex\hbox to.1em{}#1}}
\newcommand{\sect}{\Gamma}
\newcommand{\rsect}{\mathrm{R}\Gamma}

\newcommand{\Rb}{{\rm b}}

\newcommand{\SSi}{\on{SS}}

\newcommand{\opb}[1]{#1^{-1}}

\DeclareMathOperator{\ori}{or}
\DeclareMathOperator{\chv}{char}

\DeclareMathOperator{\hchv}{hypchar}

\newcommand{\indlim}[1][]{\mathop{\varinjlim}\limits_{#1}}

\newcommand{\Pro}{\mathrm{Pro}}

\newcommand{\eqdot}{\mathbin{:=}}

\newcommand{\cl}{\colon}
\newcommand{\scbul}{{\,\raise.4ex\hbox{$\scriptscriptstyle\bullet$}\,}}

\newcommand{\ba}{\begin{array}}
\newcommand{\ea}{\end{array}}

\newcommand{\bnum}{\begin{enumerate}[{\rm(i)}]}
\newcommand{\enum}{\end{enumerate}}
\newcommand{\banum}{\begin{enumerate}[{\rm(a)}]}
\newcommand{\eanum}{\end{enumerate}}

\newcommand{\eq}{\begin{eqnarray}}
\newcommand{\eneq}{\end{eqnarray}}
\newcommand{\eqn}{\begin{eqnarray*}}
\newcommand{\eneqn}{\end{eqnarray*}}

\nc{\Der}{\on{D}}

\nc{\QED}{\end{proof}}
\nc{\bA}[1][X]{\gr({\sha}_{#1})}
\nc{\Derb}{\mathrm{D}^{\mathrm{b}}}
\nc{\hs}{\hspace*}
\nc{\Supp}{\on{Supp}}
\nc{\tr}{\on{tr}}

\newcommand{\HHWo}[1][X]{\mathcal{HH}(\HWo[])}

\newcommand{\RD}{{\rm D}}
\nc{\RDAA}[1][X]{\mathrm{D}^\prime_{\mathscr{A}_{#1}}}
\nc{\RDA}{\mathrm{D}^\prime_{\mathscr{A}}}
\nc{\RDAh}{\mathrm{D}^\prime_{\mathscr{A}^{\rm loc}}}
\nc{\RDO}{\mathrm{D}^\prime_{\mathscr{O}}}
\nc{\RDOl}{\mathrm{D}^\prime_{\mathscr{O}^{\hbar,\rm loc}}}
\nc{\RDDO}{\mathrm{D}_{\mathscr{O}}}
\nc{\RDDA}{\mathrm{D}_{\mathscr{A}}}
\nc{\RDD}{\mathrm{D}^\prime}
\nc{\conv}[1][]{\mathop{\circ}\limits_{#1}}
\nc{\sconv}[1][]{\mathop{\ast}\limits_{#1}}

\nc{\ssum}{\mathop{\mbox{\normalsize$\sum$}}}
\nc{\de}[1][X]{\delta_{#1}} 
\nc{\vs}{\vspace}
\nc{\dA}[1][X]{{{\delta}_*\mspace{1mu}\mathscr{A}_{{#1}}}}
\nc{\dO}[1][X]{{{\delta_{{#1}}}_*\mspace{1mu}\mathscr{O}_{{#1}}}}
\nc{\dGA}[1][X]{\gr(\mathscr{C}_{{#1}})}

\nc{\OS}[1][X]{\sho_{#1}}
\nc{\soplus}{\mathop{\text{\scriptsize\raisebox{.5ex}{$\displaystyle\bigoplus$}}}}
\nc{\Inv}{\on{Inv}}
\nc{\stkInv}{\mathfrak{Inv}}
\nc{\bwr}{{\mbox{\large$\wr$}}}

\nc{\be}{\begin{enumerate}}
\nc{\ee}{\end{enumerate}}
\nc{\stan}{\mathrm{stan}}
\nc{\Db}{\RD^\Rb}
\nc{\pt}{\mathrm{pt}}
\nc{\BBD}{\mathbb{D}}
\nc{\rC}{\mathrm{C}}
\nc{\scup}{\mathop{\text{\scriptsize\raisebox{.5ex}{$\displaystyle\bigcup$}}}}
\nc{\tX}{{\widetilde{X}}}
\nc{\AL}{\A[\Lambda/X]}
\nc{\ALa}{\A[\Lambda^a]}
\nc{\Gr}{\on{Gr}}
\nc{\CL}{\on{\mathrm{C}_\Lambda}}
\nc{\codim}{\on{codim}}
\nc{\Chl}{\on{char_{\Lambda}}}
\nc{\Chlo}{\on{char_{\Lambda_0}}}
\nc{\ChM}{\on{char_{M}}}
\nc{\DL}{\shd_\shl}
\nc{\DLl}{\shd_\shl^\loc}
\nc{\rmC}{\rm C}

\nc{\Zh}{\Z\forl}
\nc{\PZ}{\Pro(\Zh)}
\nc{\coco}{{cohomologically complete}}
\nc{\rpi}{{{\rm R}\pi}}
\nc{\shal}{{\sha^\loc}}

\nc{\cc}{cohomologically complete}

\newcommand{\fut}[1]{{#1}^{\uparrow}}

\newcommand{\pas}[1]{{#1}^{\downarrow}}

\nc{\Cd}{\shc}

\begin{document}

\title{Hyperbolic systems and propagation on causal manifolds}
\author{Pierre Schapira%
} 
\maketitle

\begin{abstract}
In this paper, which is essentially a survey, we solve  the global Cauchy problem on causal manifolds for hyperbolic 
systems of linear partial differential equations in the framework of hyperfunctions. 
 Besides the classical  Cauchy-Kowalevsky theorem, our proofs
only use tools and ideas from the microlocal theory of sheaves of~\cite{KS90}, that is, of purely
algebraic and geometric nature. 
The study of hyperbolic $\shd$-modules is only sketched in loc.\ cit.\  and the 
global propagation results are mainly extracted from~\cite{DS99}. 
\end{abstract}

\section{Introduction}
The study of the Cauchy problem in the framework of distributions theory is
extremely difficult  and there is no 
characterization of the class of differential operators for which the problem is well posed, 
although some particular situations are well understood: operators with
constant coefficients or operators with simple characteristics
(see~\cite{Ho83}). To make an history of this subject is out of the scope of
this paper and we shall only quote J.~Leray~\cite{Le53}. 

If one replaces the sheaf of distributions by the sheaf of Sato's
hyperfunctions, the situation drastically simplifies and the Cauchy problem in
this setting was solved in~\cite{BS73} for a single differential operator and 
 in~\cite{KS79} for microfunctions solutions of microdifferential systems. 
The main difference between distribution and hyperfunction solutions  is that 
for hyperfunctions the situation is governed by the 
principal symbol of the operator (or the characteristic variety of the
system), contrary to the case of distributions. 

However, following~\cite{KS90}, a new idea has emerged: one can treat
hyperfunction solutions of  linear partial differential equations (LPDE) from
a purely sheaf theoretical point of view, the only analytic tool being the
classical Cauchy-Kowalevsky theorem. This idea is developed all along this
book and is applied in particular to the study of hyperbolic systems, but this
study is performed in a very general setting, with emphasis on the microlocal
point of view (see loc.\ cit.\ Prop.~11.5.8) and we think it may be useful to
give a more direct and elementary approach to hyperbolic systems. 

In this paper, we shall show how to solve the Cauchy problem and to treat the
propagation of solutions for general systems of LPDE in the framework of
hyperfunctions by using  the Cauchy-Kowalevsky theorem (and its extension to
systems by Kashiwara~\cite{Ka70}) and some tools from the microlocal theory of
sheaves of~\cite{KS90}. Namely, we shall use the microsupport of sheaves, the
fact that the microsupport of the complex of holomorphic solutions of a
$\shd$-module is contained in the characteristic variety of the $\shd$-module
and a theorem which gives a bound to the microsupport of the restriction to a submanifold of a sheaf. 
 
Hence, after recalling first some elementary facts of differential and
symplectic geometry, we shall recall the notion of a system of LPDE on a
complex manifold $X$ (that is, a $\shd_X$-module), the properties of its
characteristic variety and the Cauchy-Kowalevsky-Kashiwara theorem
(see~\cite{Ka70}). Next we introduce the microsupport of sheaves on a real
manifold $M$ following ~\cite{KS90}, the natural tool to describe
phenomena of propagation. We state the theorem which, given a sheaf $F$ on $M$
and a submanifold $N$ of $M$, gives a bound to the
microsupport of the sheaves $F\vert_N$ and $\rsect_NF$. 
Then we study LPDE on real manifolds, define the hyperbolic characteristic
variety and state the main theorems: one can solve  the Cauchy problem  for
hyperfunction solutions of hyperbolic systems and such solutions propagate in
the hyperbolic directions.

Finally, we study global propagation  on causal manifolds
following~\cite{DS99}.  We call here a causal manifold a pair 
$(M,\lambda)$ where $M$ is a smooth connected manifold and $\lambda$ is a
closed convex proper cone of the cotangent bundle with non empty interior at
each $x\in M$ and such that the order relation $\preceq$  associated with
$\lambda$ (by considering oriented curves whose tangents belong to the polar
cone to $\lambda$)  is closed and proper. As an immediate application of our
results, we find that if $P$ is a differential operator for which the non-zero
vectors of $\lambda$ are hyperbolic,  then $P$ induces an isomorphism on the
space $\sect_A(M;\shb_M)$  of hyperfunctions on $M$ supported by a closed set $A$ as soon as 
$A\not=M$ and $A$ is past-like. 

As mentioned in the abstract,  the study of hyperbolic $\shd$-modules is only sketched in~\cite{KS90} and 
this is the reason of this paper.

\section{Basic geometry}
In this section, we recall some elementary facts of differentiable and symplectic geometry.

\subsubsection*{Normal and conormal bundles}
Let $M$ be a real (or complex) manifold. We denote by $\tau\cl TM\to M$ its tangent bundle and by 
$\pi\cl T^*M\to M$ its cotangent bundle. If $N$ is a submanifold of $M$ we have the exact sequences of vector bundles on $N$:
\eqn
&&0\to TN\to N\times_MTM\to T_NM\to 0,\\
&&0\to T^*_NM\to N\times_MT^*M\to T^*N\to 0.
\eneqn
The vector bundle $T_NM$ is called the normal bundle to $N$ in $M$ and the vector bundle $T^*_NM$ is called the conormal bundle to $N$ in $M$. 
In the sequel, we shall identify  $M$  to  $T^*_MM$, the zero-section of $T^*M$.

\subsubsection*{Normal cones}
Let $M$ be a real manifold and let $S,Z$ be two subsets of $M$. The {\em normal cone} $C(S,Z)$ 
is a closed conic subset of $TM$ defined as follows.
Choose a local coordinate system  in a neighborhood of $x_0\in M$. Then
\eqn
&&\left\{\parbox{70ex}{
$v\in T_{x_0}M$ belongs to $C_{x_0}(S,Z)\subset T_{x_0}M$ if and only if\\
there exist sequences $\{(x_n,y_n,\lambda_n)\}_n\subset S\times Z\times\R_{>0}$ such that\\
$x_n\to[n]x_0$, $y_n\to[n]x_0$, $\lambda_n(x_n-y_n)\to[n]v$.
}\right.
\eneqn
The projection of $C(S,Z)$ on $M$ is the set $\ol S\cap\ol Z$. If $Z=\{x\}$, one writes $C_{\{x\}}(S)$ instead of $C(S,Z)$.
This is a closed cone of $T_xM$, the set of limits when $y\in S$ goes to $x$ of half-lines issued at $x$ and passing through $y$.
More generally, assume that $N$ is a smooth closed submanifold of $M$.  
At each $x\in N$, the normal cone $C_x(Z,N)$ is empty or 
contains $T_xN$. The image of $C(Z,N)$ in the quotient bundle $T_NM$ is denoted by $C_N(Z)$. 

\subsubsection*{Cotangent bundle}
Let $M$ be a real (or complex) manifold. 
The manifold $T^*M$ is a {\em homogeneous symplectic} manifold, that is, it is endowed with a canonical $1$-form $\alpha_M$, called the Liouville form, such that $\omega_M=d\alpha_M$ is a symplectic form, that is, a closed non-degenerate $2$-form.
In a local coordinate system $x=(x_1,\dots,x_n)$, 
\eqn
&&\alpha_M=\sum_{j=1}^n \xi_j\,dx_j,\quad \omega_M=\sum_{j=1}^nd\xi_j\wedge dx_j.
\eneqn
The $2$-form $\omega_M$ defines an isomorphism 
$H\cl TT^*M\simeq T^*T^*M$ called the Hamiltonian isomorphism.

\subsubsection*{Normal cones in a cotangent bundle}
Consider the particular case of a smooth Lagrangian submanifold $\Lambda$ of a cotangent bundle $T^*M$. 
The Hamiltonian isomorphism $TT^*M\isoto T^*T^*M$ induces an isomorphism
\eq\label{eq:normalcotang}
&& T_\Lambda T^*M\isoto T^*\Lambda.
\eneq
On the other-hand, consider a vector bundle $\tau\cl E\to M$. It gives rise to a morphism of vector bundles over $M$,
 $\tau'\cl TE\to E\times_MTM$ which by duality gives the map
\eq\label{eq:taud}
&&\tau_d\cl E\times_MT^*M\to T^*E.
\eneq
By restricting to the zero-section of $E$,  we get the map:
\eqn
&& T^*M\into T^*E.
\eneqn
Now consider the case of a closed submanifold $N\into M$. Using~\eqref{eq:taud} with $E=T^*_NM$, we get an embedding
$T^*N\into T^*T^*_NM$ which, using~\eqref{eq:normalcotang}, gives the embedding
\eq\label{eq:T*MintoT}
&&T^*N\into T_{T^*_NM}T^*M.
\eneq
If we choose local coordinates $(x,y)$ on $M$ such that $N=\{y=0\}$ and if one denotes by $(x,y;\xi dx,\eta dy)$ the associated symplectic coordinates on $T^*M$, then $T^*_NM=\{(x,y;\xi,\eta);y=\xi=0\}$. Denote by $(x,v\partial_y,w\partial_\xi,\eta dy)$ the associated coordinates on $T_{T^*_NM}T^*M$. Then the embedding $T^*N\into T_{T^*_NM}T^*M$ is described by
$(x;\xi)\mapsto (x,0;\xi,0)$.

\section{Linear partial differential equations}

References are made to \cite{Ka03}.

Let $X$ be a complex manifold.
One denotes by $\shd_X$  the sheaf of rings of holomorphic (finite order) differential
operators. A system of linear differential equations on $X$ is   a 
left coherent $\shd_X$-module $\shm$. The link with the intuitive notion of a
system of linear differential equations is as follows.
Locally on $X$, $\shm$ may be represented as the cokernel 
of a matrix $\cdot P_0$ of differential operators acting on the right: 
\eqn
&&\shm\simeq \shd_X^{N_0}/\shd_X^{N_1}\cdot P_0.
\eneqn
By classical arguments of analytic geometry (Hilbert's syzygies
theorem), one shows that 
$\shm$ is locally isomorphic to the cohomology of a bounded complex
$$\shm^\bullet\eqdot
0\to \shd_X^{N_r}\to\cdots\to\shd_X^{N_1}\to[\cdot P_0]\shd_X^{N_0}\to 0.$$
The complex of holomorphic solutions of $\shm$, denoted $\Sol(\shm)$,
(or better in the language of derived categories,
$\rhom[\shd_X](\shm,\sho_X)$), is obtained by applying 
$\hom[\shd_X](\cdot,\sho_X)$ to $\shm^\bullet$. Hence
\eq\label{eq:solm}
&&\Sol(\shm)\simeq
0\to \sho_X^{N_0}\to[P_0\cdot]\sho_X^{N_1}\to\cdots\sho_X^{N_r}\to 0,
\eneq
where now $P_0\cdot$ operates on the left.

One defines naturally the characteristic variety of $\shm$, 
denoted $\chv(\shm)$, a closed complex analytic subset of $T^*X$, 
conic with respect to the action of $\C^\times$ on $T^*X$.
For example, if $\shm$ has a single generator $u$ with relation $\shi u=0$, where
$\shi$ is a locally finitely generated left ideal of $\shd_X$, then 
\eqn
&&\chv(\shm)=\{(z;\zeta)\in T^*X; \sigma(P)(z;\zeta)=0\mbox{ for all }P\in\shi\},
\eneqn
where $\sigma(P)$ denotes the principal symbol of $P$.

The fundamental  result below was obtained in~\cite{SKK73}.
\begin{theorem}
Let $\shm$ be a coherent $\shd_X$-module. Then 
$char(\shm)$ is a closed conic  complex analytic {\em involutive} \lp {\em i.e.,} co-isotropic\rp\, subset of $T^*X$. 
\end{theorem}
The proof  of the involutivity is really difficult:  it uses microdifferential operators of infinite order and quantized contact transformations. 
Later, Gabber \cite{Ga81} gave a purely algebraic (and much simpler) proof of this result.

\subsubsection*{Cauchy problem for LPDE}
Let $Y$ be a complex submanifold of the complex manifold $X$ and let $\shm$ be a coherent $\shd_X$-module. 
One can define the induced $\shd_Y$-module $\shm_Y$, but in general it is an object of the derived category
$\Derb(\shd_Y)$  which is neither concentrated in degree zero nor coherent. Nevertheless, there is a natural morphism
\eq\label{eq:CKK}
&&\rhom[\shd_X](\shm,\sho_X)\vert_Y\to\rhom[\shd_Y](\shm_Y,\sho_Y).
\eneq
Recall that one says that $Y$ is non-characteristic for $\shm$ if 
\eqn
&&\chv(\shm)\cap T^*_YX\subset T^*_XX.
\eneqn
With this hypothesis, the induced system $\shm_Y$ by $\shm$ on $Y$ is
a coherent $\shd_Y$-module and one has the Cauchy-Kowalesky-Kashiwara theorem~\cite{Ka70}:

\begin{theorem}\label{th:CKK}
Assume $Y$ is non-characteristic for $\shm$. Then $\shm_Y$ is a coherent $\shd_Y$-module and the morphism~\eqref{eq:CKK}
is an isomorphism.
\end{theorem}

\begin{example}\label{exa:CK}
Assume $\shm=\shd_X/\shd_X\cdot P$ for a differential operator $P$ of order $m$ and $Y$ is a hypersurface.
In this case, the induced system $\shm_Y$ is isomorphic to $\shd_Y^m$ and one recovers the classical 
 Cauchy-Kowalesky theorem.
 
 More precisely, choose a local coordinate system $z=(z_0,z_1,\dots,z_n)=(z_0,z')$ 
 on $X$ such that $Y=\{z_0=0\}$. 
Then $Y$ is non-characteristic with
respect to $P$ ({\em i.e.,} \/for the $\shd_X$-module
$\shd_X/\shd_X\cdot P$) 
if and only if $P$ is written as 
\eq
&&P(z_0,z';\partial_{z_0},\partial_{z'})=
\sum_{0\leq j\leq m}a_j(z_0,z',\partial_{z'})\partial_{z_0}^j
\eneq
where $a_j(z_0,z',\partial_{z'})$ is a differential operator not depending
on $\partial_{z_0}$ of order $\leq m-j$ and $a_m(z_0,z')$ (which is a holomorphic
function on $X$) satisfies: $a_m(0,z')\neq 0$. 
By the definition of  the induced system $\shm_Y$ we obtain
\eqn
&&\shm_Y\simeq \shd_X/(z_0\cdot\shd_X+\shd_X\cdot P).
\eneqn
 By the Sp\"ath-Weierstrass division theorem for differential operators,  any $Q\in\shd_X$ may be written uniquely in a neighborhood of $Y$ as
\eqn
&&Q= R\cdot P+\sum_{j=0}^{m-1}S_j(z,\partial_{z'})\partial_{z_0}^j,
\eneqn
hence, as
\eqn
&&Q= z_0\cdot Q_0+R\cdot P+\sum_{j=0}^{m-1}R_j(z',\partial_{z'})\partial_{z_0}^j.
\eneqn
Therefore $\shm_Y$ is isomorphic to $\shd_Y^m$. Theorem~\ref{th:CKK} gives:
\eqn
&&\hom[\shd_X](\shm,\sho_X)\vert_Y\simeq\sho_Y^m,\quad \ext[\shd_X]{1}(\shm,\sho_X)\vert_Y\simeq0.
\eneqn
In other words, the morphism which to a holomorphic solution $f$  of the homogeneous equation $Pf=0$ 
associates its $m$-first traces on $Y$ is an isomorphism and one can solve the equation $Pf=g$ is a neighborhood of each point of $Y$.

This is exactly the classical Cauchy-Kowalesky theorem.
Note that the proof of Kashiwara of the general case is deduced form the classical theorem by purely algebraic arguments.
\end{example}

\section{Microsupport and propagation}

References  are made to~\cite{KS90}.

The idea of microsupport takes its origin in the study of LPDE and particularly in the study of hyperbolic systems.
Let $F$ be a sheaf on a real manifold $M$. Roughly speaking, one says that $F$ propagates in the codirection 
$p=(x_0;\xi_0)\in T^*M$ if for any open set $U$ of $M$ such that $x_0\in\partial U$, $\partial U$ is smooth in a neighborhood of $x_0$ and $\xi_0$ is the exterior normal vector to $U$ at $x_0$, any section of $F$ on $U$ extends through $x_0$, that is, extends to a bigger open set $U\cup V$ where $V$ is a neighborhood of $x_0$. 

\begin{example}\label{exa:CKprop}
(i) Assume $X$ is a complex manifold, $F$ is the sheaf of holomorphic solutions of the equation $Pf=0$ where $P$ is a differential operator and $\sigma(P)(p)\not=0$. Then $F$ propagates in the codirection $p$. This follows easily from the Cauchy-Kowalevsky theorem (see~\cite[\S~9.4]{Ho83}). 

\vspace{0.2ex}\noindent
(ii) Assume $M=\R\times N$ where $N$ is a Riemannian manifold. Let $P=\partial_t^2-\Delta$ be the wave equation. 
(Here, $t$ is the coordinate on $\R$ and $\Delta$ is the Laplace operator on $N$.)
Let $F$ be the sheaf of distribution solutions of the equation $Pu=0$. Then $F$ propagates  at each 
$p=(t_0,x_0;\pm1,0)$.
\end{example}

\subsubsection*{Microsupport}
Let $M$ denote a {\em real} manifold of class $C^\infty$, let $\cor$ be a
field, and let $F$ be a bounded complex of sheaves of $\cor$-vector
spaces on $M$ (more precisely, $F$ is an object of $\Derb(\cor_M)$, the
bounded derived category of sheaves on $M$). 

\begin{definition}
Let $F\in \Derb(\cor_M)$. The microsupport $\SSi(F)$ is the closed $\R^+$-conic subset of $T^*M$ defined as follows: 
for an open subset $W\subset T^*M$ one has $W\cap\SSi(F)=\emptyset$ if and only if 
for any $x_0\in M$ and any
real $\Cd^1$-function $\phi$ on $M$ defined in a neighborhood of $x_0$ 
with $(x_0;d\phi(x_0))\in W$, one has
$(\rsect_{\{x;\phi(x)\geq\phi(x_0)\}} F)_{x_0}\simeq0$.  
\end{definition}
In other words, $p\notin\SSi(F)$ if the sheaf $F$ has no cohomology 
supported by ``half-spaces'' whose conormals are contained in a 
neighborhood of $p$. 
Note that the condition $(\rsect_{\{x;\phi(x)\geq\phi(x_0)\}} F)_{x_0}\simeq0$ is equivalent to the following:
setting $U=\{x\in M;\phi(x)<\phi(x_0)\}$, one has the isomorphism for all $j\in\Z$
\eqn
\indlim[V\ni x_0]H^j(U\cup V;F)\isoto H^j(U;F).
\eneqn

\begin{itemize}
\item
By its construction, the microsupport is $\R^+$-conic, that is,
invariant by the action of  $\R^+$ on $T^*M$. 
\item
$\SSi(F)\cap T^*_MM=\pi(\SSi(F))=\Supp(F)$.
\item
The microsupport satisfies the triangular inequality:
if $F_1\to F_2\to F_3\to[{\;+1\;}]$ is a
distinguished triangle in  $\Derb(\cor_M)$, then 
$\SSi(F_i)\subset\SSi(F_j)\cup\SSi(F_k)$ for all $i,j,k\in\{1,2,3\}$
with $j\not=k$. 
\end{itemize}
In the sequel, for a locally closed subset $A\subset M$, we denote by $\cor_{A}$ the sheaf on $M$ which is the constant sheaf 
with stalk $\cor$ on $A$ and is zero on $M\setminus A$. 

\begin{example}\label{ex:microsupp}
(i) If $F$ is a non-zero local system on $M$ and $M$ is connected, then $\SSi(F)=T^*_MM$.

\noindent
(ii) If $N$ is a closed submanifold of $M$ and $F=\cor_N$, then 
$\SSi(F)=T^*_NM$, the conormal bundle to $N$ in $M$.

\noindent
(iii) Let $\phi$ be a $\Cd^1$-function such that $d\phi(x)\not=0$ 
whenever $\phi(x)=0$.
Let $U=\{x\in M;\phi(x)>0\}$ and let $Z=\{x\in M;\phi(x)\geq0\}$. 
Then 
\eqn
&&\SSi(\cor_U)=U\times_MT^*_MM\cup\{(x;\lambda d\phi(x));\phi(x)=0,\lambda\leq0\},\\
&&\SSi(\cor_Z)=Z\times_MT^*_MM\cup\{(x;\lambda d\phi(x));\phi(x)=0,\lambda\geq0\}.
\eneqn
\end{example}
For a precise definition of being co-isotropic,
we refer to~\cite[Def.~6.5.1]{KS90}.

\begin{theorem}\label{th:coisotr}
Let $F\in \Derb(\cor_M)$. Then its microsupport 
$\SSi(F)$ is co-isotropic.
\end{theorem}

\subsubsection*{Microsupport and characteristic variety}
Assume now that $(X,\sho_X)$ is a complex manifold  and let $\shm$ be a coherent $\shd_X$-module.  Recall that one sets for short 
$\Sol(\shm)\eqdot\rhom[\shd_X](\shm,\sho_X)$ (see~\eqref{eq:solm}).

After identifying $X$  with its real underlying manifold,
the link between the microsupport of sheaves and the characteristic
variety of coherent $\shd$-modules is given  by:

\begin{theorem}\label{th:ssinchar1}{\rm (See~\cite[Th.~11.3.3]{KS90}.)}
Let $\shm$ be a coherent $\shd_X$-module. Then 
$\SSi(\Sol(\shm))=\chv(\shm)$.
\end{theorem}
The inclusion $\SSi(\Sol(\shm))\subset \chv(\shm)$
is the most useful in practice. By purely algebraic arguments one reduces its proof to the case where 
$\shm=\shd_X/\shd_X\cdot P$, in which case this result is due to  Zerner~\cite{Ze71} who deduced it from 
the Cauchy-Kowalevsky theorem in its precise form given by Petrovsky and Leray 
(see also~\cite[\S~9.4]{Ho83}). As a corollary of Theorems~\ref{th:coisotr} and~\ref{th:ssinchar1}, one recovers the fact that the characteristic variety 
of a coherent $\shd_X$-module is co-isotropic.

\subsubsection*{Propagation 1}

Consider a closed submanifold $N$ of $M$  and let $F\in\Derb(\cor_M)$. 
There is a natural morphism
\eq\label{eq:sectNvertN}
&& F\vert_N\to\rsect_NF\tens\ori_{N/M}\,[d].
\eneq
Here, $\ori_{N/M}$ is the relative orientation sheaf and $d$ is the codimension of $N$. To better understand this morphism, consider the case where $N$ is a hypersurface dividing $M$ into two closed half-spaces $M^+$ and $M^-$. Then we have a distinguished triangle
\eq\label{eq:dtNM}
&&\rsect_NF\to[\alpha] (\rsect_{M^+}F)\vert_N\oplus(\rsect_{M^-}F)\vert_N\to[\beta] F\vert_N\to[+1].
\eneq
Here $\alpha(u)=(u,-u)$ and $\beta(v,w)=v+w$, but one can also replace the morphism $\alpha$ with $-\alpha$. If one wants 
morphisms intrinsically defined, then one way is to replace $\rsect_NF$ with $\rsect_NF\tens\ori_{N/M}$. 

The next result will be used when studying the Cauchy problem for hyperfunctions.

\begin{theorem}\label{th:divCK} {\rm (See~\cite[Cor.~5.4.11]{KS90})}
Assume that $\SSi(F)\cap T^*_NM\subset T^*_MM$. Then the morphism~\eqref{eq:sectNvertN} is an isomorphism.
\end{theorem}
When $N$ is a hypersurface, the proof is obvious by~\eqref{eq:dtNM} since it follows from the definition of the microsupport and
the hypothesis on $F$ that $(\rsect_{M^\pm}F)\vert_N\simeq0$. To treat the general case, one uses the Sato's microlocalization functor (see~\cite{SKK73} or~\cite{KS90}). 

\subsubsection*{Propagation 2}
Let $N\into M$ and $F$ be as above. 
A natural question is to calculate, or at least to give a bound, to the microsupport of the restriction $F\vert_N$ or to $\rsect_NF$. 
The answer is given by 

\begin{theorem}\label{th:hyp1}{\rm (See~\cite[Cor.~6.4.4]{KS90}.)}
One has
\eqn
&&\SSi(\rsect_NF)\subset T^*N\cap C_{T^*_NM}(\SSi(F)),\\
&&\SSi(F\vert_N)\subset T^*N\cap C_{T^*_NM}(\SSi(F)).
\eneqn
\end{theorem}
Recall that $T^*N$ is embedded into $T_{T^*_NM}T^*M$ by~\eqref{eq:T*MintoT} and the normal cone 
$C_{T^*_NM}(\SSi(F))$ is a closed subset of $T_{T^*_NM}T^*M$.

\section{Hyperbolic systems}
References are made to~\cite{KS90}.

In this section we denote by $M$ a real analytic manifold of dimension $n$ and by $X$ a complexification of $M$. 
When necessary, we shall identify the complex manifold $X$ with  the real underlying manifold to $X$. 

\subsubsection*{Hyperfunctions}
We have the sheaves
\eq\label{eq:sato1}
&&\sha_M=\sho_X\vert_M,\quad \shb_M=H^n_M(\sho_X)\tens\ori_M.
\eneq
Here, $\ori_M$ is the orientation sheaf on $M$. The sheaf $\sha_M$ is the sheaf of real analytic 
functions on $M$ and the sheaf $\shb_M$ is the sheaf of Sato's hyperfunctions on $M$ (\cite{Sa59}). It is a flabby sheaf and it contains the sheaf of distributions on $M$ as a subsheaf. Moreover, the cohomology objects $H^j_M(\sho_X)$ are zero for
 $j\not=n$ and therefore we may better write
 \eq\label{eq:sato2}
 &&\shb_M= \rsect_M(\sho_X)\tens\ori_M\,[n].
 \eneq
This will be essential in the proofs of Theorems~\ref{th:hyp3} and~~\ref{th:hyp4} below.

\subsubsection*{Propagation for hyperbolic systems}

\begin{definition}
Let $\shm$ be a coherent left $\shd_X$-module.
We set
\eqn
&&\hchv_M(\shm)=T^*M\cap C_{T^*_MX}(\chv(\shm))
\eneqn
and call $\hchv_M(\shm)$ the hyperbolic characteristic variety of $\shm$ along $M$. 
A vector $\theta\in T^*M$ such that $\theta\notin\hchv_M(\shm)$ is called hyperbolic with respect to $\shm$. 
In case $\shm=\shd_X/\shd_X\cdot P$ for a differential operator $P$, one says that $\theta$ is hyperbolic for $P$.
\end{definition}
\begin{example}
Assume we have a local coordinate system $z=x+\sqrt{-1}y$ and $M=\{y=0\}$. Denote by $(z;\zeta)$ the symplectic coordinates on $T^*X$ with $\zeta=\xi+\sqrt{-1}\eta$. Let $(x_0;\theta_0)\in T^*M$ with $\theta_0\not=0$. Let $P$ be a differential operator with principal symbol 
$\sigma(P)$. Applying the definition of the normal cone, we find that $(x_0;\theta_0)$ is hyperbolic for $P$ 
 if and only if 
\eq\label{eq:thetahyp}
&&\left\{\parbox{65ex}{
there exist an open neighborhood $U$ of $x_0$ in $M$ and an open  conic neighborhood $\gamma$ of 
 $\theta_0\in\R^n$ such that 
 $\sigma(P)(x;\theta+\sqrt{-1}\eta)\not=0$ 
 for all $\eta\in\R^n$, $x\in U$ and $\theta\in\gamma$.
 }\right.
 \eneq
As noticed by M.~Kashiwara, it follows from the local Bochner's tube theorem that  
condition~\eqref{eq:thetahyp}
will be satisfied as soon as  $\sigma(P)(x;\theta_0+\sqrt{-1}\eta)\not=0$ for all $\eta\in\R^n$ and $x\in U$ (see~\cite{BS73}). Hence, one recovers the classical notion of a (weakly) hyperbolic operator (see~{Le53}).
 \end{example}
 
\begin{theorem}\label{th:hyp3}
Let $\shm$ be a coherent $\shd_X$-module. Then
\eqn
&&\SSi(\rhom[\shd_X](\shm,\shb_M))\subset\hchv_M(\shm).
\eneqn
The same result holds with $\sha_M$ instead of $\shb_M$.
\end{theorem}
\begin{proof}
This follows from Theorem~\ref{th:hyp1} and the isomorphisms
\eqn
&&\rsect_M\rhom[\shd_X](\shm,\sho_X)\simeq\rhom[\shd_X](\shm,\rsect_M\sho_X),\\
&&\rhom[\shd_X](\shm,\sho_X)\vert_M\simeq\rhom[\shd_X](\shm,\sho_X\vert_M).
\eneqn
\end{proof}

\subsubsection*{Cauchy problem for hyperbolic systems}
We consider the following situation: $M$ is a real analytic manifold of
dimension $n$, $X$ is a complexification of $M$, 
$N\into M$ is a real analytic smooth closed submanifold of $M$ of codimension
$d$ and $Y\into X$ is a complexification of $N$ in $X$. 

\begin{theorem}\label{th:hyp4}
Let $M,X,N,Y$ be as above and let $\shm$ be a coherent $\shd_X$-module.
We assume 
\eq\label{hyp:nonhyp}
&& T^*_NM\cap\hchv_M(\shm)\subset T^*_MM. 
\eneq
In other words, any non zero vector $\theta\in T^*_NM$ is hyperbolic for
$\shm$. Then 
$Y$ is non characteristic for $\shm$ in a neighborhood of $N$ and
the  isomorphism~\eqref{eq:CKK} induces the isomorphism
\eq\label{eq:CKhyp}
&&\rhom[\shd_X](\shm,\shb_M)\vert_N\isoto \rhom[\shd_Y](\shm_Y,\shb_N).
\eneq
\end{theorem}
\begin{proof}
(i) Since $X$ is a complexification of $M$, there is an isomorphism
$M\times_XT^*X\simeq T^*_MX\oplus_M T^*M$. Moreover, there is a natural
embedding $T^*_MX\oplus_M T^*M\hookrightarrow T_{T^*_MX}T^*X$
(see~\cite[\S~6.2]{KS90}). Then hypothesis~\eqref{hyp:nonhyp}
implies $T^*_NM\cap M\times_X\chv(\shm)\subset T^*_XX$ and since $\chv(\shm)$
is $\C^\times$-conic,  $N\times_YT^*_YX\cap\chv(\shm)\subset T^*_XX$. Hence,
$Y$ is non characteristic for $\shm$. 

\vspace{0.2ex}\noindent
(ii) We have the chain of isomorphisms
\eqn
\rhom[\shd_X](\shm,\shb_M)\vert_N&\simeq&\rsect_N\rhom[\shd_X](\shm,\shb_M)\tens\ori_{N/M}\,[d]\\
&\simeq&\rsect_N\rhom[\shd_X](\shm,\rsect_M\sho_X)\tens\ori_{N}\,[n+d]\\
&\simeq&\rsect_N\rhom[\shd_X](\shm,\rsect_Y\sho_X)\tens\ori_{N}\,[n+d]\\
&\simeq&\rsect_N\rhom[\shd_X](\shm,\sho_X)\vert_Y\tens\ori_{N}\,[n-d]\\
&\simeq&\rsect_N\rhom[\shd_Y](\shm_Y,\sho_Y)\tens\ori_{N}\,[n-d]\\
&\simeq&\rhom[\shd_Y](\shm_Y,\shb_N).
\eneqn
Here, the first isomorphism follows from Theorems~\ref{th:hyp3} and~\ref{th:divCK}, the second uses the definition of the sheaf $\shb_M$, 
the third is obvious since $N$ is both contained in $M$ and in $Y$, the fourth follows from 
Theorems~\ref{th:ssinchar1} and~\ref{th:divCK},
the fifth is Theorem~\ref{th:CKK} and the last one uses the definition of the sheaf $\shb_N$.
\end{proof}
Consider for simplicity the case where $\shm=\shd_X/\shi$ where $\shi$ is a coherent left ideal of $\shd_X$. A section 
$u$ of $\hom[\shd_X](\shm,\shb_M)$ is a hyperfunction $u$ such that $Qu=0$ for all $Q\in\shi$. 
It follows that the analytic wave front set of $u$ does not intersect $T^*_YX\cap T^*_MX$ and this implies that the restriction of 
$u$ (and its derivative) to $N$ is well-defined as a hyperfunction on $N$. One can show that the morphism~\eqref{eq:CKhyp} is then obtained using this restriction morphism, similarly as in Theorem~\ref{th:CKK}. Since we do not recall the Sato's microlocalization and the notion of  wave front set in this paper,  we do not explain this point.

Note that Theorem~\ref{th:hyp3} and~\ref{th:hyp4} were first obtained in~\cite{BS73} in case of a single differential operator and in~\cite{KS79} in the more general situation of a systems of microdifferential operators acting on microfunctions.

\section{Global propagation on causal manifolds}
There is a vast literature on  Lorentzian manifolds (for an exposition, see {\em e.g.,} the book~\cite{BGP07}) but we shall restrict ourselves to recall a  global propagation theorem of~\cite{DS99} and give applications. 

\subsubsection*{Global propagation for sheaves}
For a manifold $M$ we denote by $q_1$ and $q_2$ the first and second projection defined on $M\times M$, by 
 $q_{ij}$ the $(i,j)$-th projection   defined on  $M\times M\times M$ and similarly on  $M\times M\times M\times M$. We denote by $\Delta_M$ the diagonal of $M\times M$.
 
 A cone $\lambda$ in a vector bundle $E\to M$ is a subset of $E$ which is invariant by the action of $\R^+$ on this vector bundle. 
We denote by $\lambda^a$ the opposite cone to $\lambda$, that is, $\lambda^a=-\lambda$  and by $\lambda^\circ$ the polar cone to $\lambda$, a closed convex cone of the dual vector bundle
\eqn
&&\lambda^\circ=\{(x,\xi)\in E^*;\langle \xi,v\rangle\geq0\mbox{ for all }v\in\lambda\}.
\eneqn

In all this section, we assume that $M$ is connected. 
\begin{definition}\label{def:Zproper}
Let $Z$ be a closed subset of $M\times M$ and let $A$ be a closed subset of $M$. We say that 
$A$ is $Z$-proper if $q_1$ is proper on $Z\cap\opb{q_2}(A)$.
\end{definition}

\begin{definition}{\rm (See~\cite[Def.~1.2]{DS99}.)}\label{def:propagator}
A convex propagator $(Z,\lambda)$ is the data of a closed subset $Z$ of $M\times M$ and 
 a closed convex proper cone $\lambda$ of $T^*M$ satisfying
\eq\label{hyp:causal2}
&&\left\{ \parbox{65ex}{
(i) $\Delta_M\subset Z$, \\
(iii) $\SSi(\cor_Z)\cap (T^*M\times T^*_MM\cup T^*_MM\times T^*M)\subset T^*_{M\times M}M\times M$,\\
(iii) $\SSi(\cor_Z)\subset T^*M\times\lambda$.
}\right.
\eneq
\end{definition}

\begin{theorem}{\rm (See~\cite[Cor.~1.4]{DS99}.)}\label{th:DSpropa}
Let $(Z,\lambda)$ be a convex propagator and let $A$ be a $Z$-proper closed subset of $M$ with $A\not=M$. 
Let $F\in\Derb(\cor_M)$ and assume that $\SSi(F)\cap\lambda^a\subset T^*_MM$ and $\SSi(\cor_A)\subset\lambda^a$.
Then $\rsect_A(M;F)\simeq0$.
\end{theorem}
Note that the conclusion of the theorem is equivalent to saying that we have the isomorphism 
$\rsect(M;F)\isoto\rsect(M\setminus A;F)$. Roughly speaking, the ``sections'' of $F$ on $M\setminus A$ extend uniquely to $M$. 

\subsubsection*{Causal manifolds}

In the literature, one often encounters  time-orientable Lorentzian manifolds to which one can associate a cone in $TM$ or 
its polar cone in $T^*M$.  Here, we only assume that:
\eq\label{hyp:lorentz}
&&\left\{\parbox{65ex}{
$M$ is a smooth real connected manifold and 
we are given a closed convex proper cone $\lambda$ in $T^*M$ 
such that for each $x\in $M$, \Int(\lambda_x)\not=\emptyset$.
}\right.
\eneq

\begin{definition}
A $\lambda$-path is a continuous piecewise  $C^1$-curve $\gamma\cl [0,1]\to M$ such that its derivative 
$\gamma'(t)$ satisfies $\langle \gamma'(t),v\rangle\geq0$ for all $t\in[0,1]$ and $v\in\lambda$. 
Here $\gamma'(t)$ means as well the right or the left derivative, as soon as it exists (both exist on $]0,1[$ and are almost everywhere the same, and $\gamma'_r(0)$ and $\gamma'_l(1)$ exist).
\end{definition}
To $\lambda$ one associates a preorder on $M$ as follows: $x\preceq y$ if and only if there exists a $\lambda$-path
$\gamma$ such that $\gamma(0)=x$ and $\gamma(1)=y$. 

For a subset $A$ of $M$, we set:
\eqn
\pas A &=& \{x\in M;\mbox{ there exists } y\in A, x\preceq y\},\\
\fut A &=& \{x\in M;\mbox{ there exists } y\in A, y\preceq x\}.
\eneqn
We shall assume:
\eq\label{hyp:preorder}
&&\left\{\parbox{65ex}{
the relation $\preceq$ is closed and proper, that is, \\
(i) if $\{(x_n,y_n)\}_n $ is a sequence which converges to $(x,y)$  
and $x_n\preceq y_n$ for all $n$, then $x\preceq y$,\\
(ii) for two compact sets $A$ and $B$, the set $\fut B \cap \pas A$ is compact.
}\right.
\eneq
\begin{definition}
A pair $(M,\lambda)$ with $\lambda\subset T^*M$ satisfying~\eqref{hyp:lorentz} and~\eqref{hyp:preorder} will be called here  
{\em a causal manifold}.
\end{definition}
Note that if $(M,\lambda)$ is a causal manifold, then so is $(M,\lambda^a)$.

One denotes by $Z_\lambda$ the set of $M\times M$ associated with the preorder:
\eqn
&&Z_\lambda=\{(x,y)\in M\times M;x\preceq y\}.
\eneqn
Note that giving a relation $\preceq$ satisfying~\eqref{hyp:preorder} is equivalent to giving $Z_\lambda$ satisfying:
\eq\label{hyp:Zlambda}
&&\left\{\parbox{65ex}{
$\Delta_M\subset Z_\lambda$,\\
$q_{13}(\opb{q_{12}}Z_\lambda\cap \opb{q_{12}}Z_\lambda)\subset Z_\lambda$,\\
$Z_\lambda$ is closed,\\
$q_{13}$ is proper on $\opb{q_{12}}Z_\lambda\cap \opb{q_{23}}Z_\lambda$. 
}\right.
\eneq
Note that for a closed subset $A$ of $M$
\bnum
\item
$\pas A = q_{1}(Z_\lambda\cap \opb{q_{2}} A)$ and $\fut A = q_{2}(Z_\lambda\cap \opb{q_{1}} A)$,
\item
if $A$ is compact, the map $q_1$ is proper on $Z_\lambda\cap\opb{q_2}A$ (since $Z_\lambda$ is closed),
\item
if $A$ is compact, then $\pas A$ is closed (by (i) and (ii)),
\item
for two compact sets $A$ and $B$, the set $(\fut B \times \pas A)\cap Z_\lambda$ is compact (indeed, this set
is contained in $(\fut B \cap \pas A)\times(\fut B \cap \pas A)$),
\item
$A$ is $Z_\lambda$-proper if and only if, for any compact set $B$, the set
$\fut B\cap A$ is compact. In particular, if $A$ is compact, then $\pas A$ is $Z_\lambda$-proper.
\enum

\begin{proposition}{\rm (See~\cite[Prop.~4.4]{DS99}.)}
Let $(M,\lambda)$ be a causal manifold. Then 
\banum
\item
$(Z_\lambda,\lambda)$ is a convex propagator,
\item
if $A$ is a closed subset satisfying $\pas A =A$, then $\SSi(\cor_A)\subset\lambda^a$.
\eanum
\end{proposition}
In particular, if $A$ is a closed subset such that $\pas A=\fut A$, then $\SSi(\cor_A)\subset T^*_MM$ and therefore 
$A=\emptyset$  or $A=M$. 
\begin{proof}[Sketch of proof]
To a set $A\subset M$, one associates its strict normal cone $N(A)$  (\cite[Def.~5.3.6]{KS90}), an open convex cone of $TM$.  In a local coordinate system, 
$(x_0;v_0)\in N(A)$ if and only if there exists an open cone $\gamma$ containing $v_0$ and an open neighborhood $U$ of $x_0$ such that 
\eqn
&&U\cap\bl(A\cap U)+\gamma\br\subset A.
\eneqn
One shows that the hypothesis $\pas A =A$ implies that $\Int(\lambda^{\circ a})\subset N(A)$. 
Then the proof  of (b) follows from the inclusion 
$\SSi(\cor_A)\subset N(A)^\circ$ (\cite[Prop.~5.3.8]{KS90}). 

The proof of (a) is similar.
\end{proof}

We can reformulate Theorem~\ref{th:DSpropa} as follows.
\begin{theorem}\label{th:DSpropa2} 
Let $(M,\lambda)$ be a causal manifold. 
Let $A$ be a  closed subset of $M$ such that $A=\pas A$,  $A\not=M$ and for any compact subset $B$ of $M$, the set $\fut B\cap A$ is compact. 
Let $F\in\Derb(\cor_M)$ and assume that $\SSi(F)\cap\lambda^a\subset T^*_MM$.
Then $\rsect_A(M;F)\simeq0$.
\end{theorem}
Now let us take for $F$ the complex of hyperfunction solutions of a $\shd_M$-module $\shm$. We obtain

\begin{corollary}
Let $(M,\lambda)$ and $A$ be as in {\rm Theorem~\ref{th:DSpropa2}}. 
Let $\shm$ be a coherent $\shd_X$-module and assume that $\lambda\cap\hchv_M(\shm)\subset T^*_MM$.
In other words, all non-zero vectors of $\lambda$ are hyperbolic for $\shm$. 
Then $\RHom[\shd_X](\shm,\sect_A\shb_M)\simeq0$ or equivalently
\eqn
&&\rsect(M;\rhom[\shd_X](\shm,\shb_M))\isoto\rsect(M\setminus A;\rhom[\shd_X](\shm,\shb_M)).
\eneqn
\end{corollary}

\begin{example}\label{exa:Phyp}
Let us particularize to the case of a single differential operator, that is, $\shm=\shd_X/\shd_X\cdot P$.
We find that $P$ induces an isomorphism $\sect_A(M;\shb_M)\isoto \sect_A(M;\shb_M)$. In particular
\begin{itemize}
\item 
for any $v\in \sect_A(M;\shb_M)$, there exists a unique $u\in \sect_A(M;\shb_M)$ such that $Pu=v$,
\item 
any $u\in\sect(M\setminus A;\shb_M)$ solution of $Pu=0$ extends uniquely all over $M$ as a solution of this equation.
\end{itemize}
\end{example}

\begin{corollary}
Let $(M,\lambda)$ be a causal manifold, $N$ a hypersurface which divides $M$
into two closed sets $M^+$ and $M^-$, and let $\shm$ be a coherent $\shd_X$-module.
Assume
\banum
\item
$M^\pm\not=M$,
$M^-=\pas{M^-}$, $M^+=\fut{M^+}$ and for any compact subset $B$ of $M$, 
the sets $\fut B\cap M^-$ and  $\pas B\cap M^+$ are
compact,
\item
$T^*_NM\subset\lambda\cup\lambda^a$,
\item
  $\lambda\cap\hchv_M(\shm)\subset T^*_MM$.
\eanum
Then the restriction morphism
$\RHom[\shd_X](\shm,\shb_M)\to\RHom[\shd_Y](\shm\vert_Y,\shb_N)$ 
is an isomorphism.
In other words, the Cauchy problem for hyperfunctions with data on $N$ is globally well-posed.
\end{corollary}

\begin{remark}
One shall be aware that hypothesis~\eqref{hyp:Zlambda} may be satisfied on $M$ and not on an open subset of $M$. Following~\cite[Rem.~3.1.5]{BGP07} consider $M=\R\times\R^n$ with linear coordinates $x=(x_0,x')$ and the closed proper cone 
$\lambda=\{x;\xi_0,\xi');\xi_0\geq\vert\xi'\vert\}$ of $T^*M$. It is easy to construct a convex open set $\Omega$ 
and $x,y\in\Omega$ such that $\fut {\{y\}} \cap \pas {\{x\}}\cap\Omega$ is not compact and to construct a non zero solution $u$ of the equation $Pu=0$ on $\Omega$, where $P$ is the wave equation, with support contained in  
$\pas {\{x\}}\cap\Omega$.
\end{remark}

\providecommand{\bysame}{\leavevmode\hbox to3em{\hrulefill}\thinspace}

\vspace*{1cm}
\noindent
\parbox[t]{21em}
{\scriptsize{
\noindent
Pierre Schapira\\
Institut de Math{\'e}matiques,
Universit{\'e} Pierre et Marie Curie\\
4 Place Jussieu, 7505 Paris\\
e-mail: schapira@math.jussieu.fr\\
and\\
Mathematics Research Unit, \\
University of Luxemburg\\
http://www.math.jussieu.fr/\textasciitilde schapira/
}}
\end{document}